\documentclass[11pt,centertags]{amsart}
\usepackage{amsthm}
\usepackage{amssymb}            
\usepackage{mathrsfs}             
\usepackage{hyperref}           
\usepackage{color}

\newif\ifPDF
\ifx\pdfoutput\undefined
    \PDFfalse
\else
    \ifnum\pdfoutput > 0
        \PDFtrue
    \else
        \PDFfalse
    \fi
\fi
\ifPDF
    \usepackage[pdftex]{graphicx}
    \DeclareGraphicsExtensions{.pdf,.png,.jpg}
    \graphicspath{{}}
\else
    \usepackage{graphicx}
    \DeclareGraphicsExtensions{.eps}
    \graphicspath{{}}
\fi

\newtheorem*{main*}{Main Theorem}

\newtheorem{theorem}{Theorem}[section]
\newtheorem*{theorem*}{Theorem}

\newtheorem{lemma}[theorem]{Lemma}

\newtheorem*{question*}{Question}

\newtheorem*{conjecture*}{Conjecture}

\theoremstyle{definition}

\newtheorem*{definition*}{Definition}

\theoremstyle{remark}

\numberwithin{equation}{section}

\usepackage[total={6.5in,9in},top=1.2in, left=0.9in, includefoot]{geometry}

\newcommand{\R}{\mathbb{R}}

\newcommand{\mc}{\mathcal}

\newcommand{\mf}{\mathfrak}

\newcommand{\eps}{\varepsilon }

\DeclareMathOperator{\SL}{SL}

\DeclareMathOperator{\SO}{SO}

\DeclareMathOperator{\Sp}{\Sp}

\newcommand{\op}[1]{\operatorname{#1}}
\newcommand{\set}[1]{\left\{#1 \right\}}

\newcommand{\directsum}{\oplus}

\newcommand{\rank}{\operatorname{rank}}

\newcommand{\pa}{\partial }

\providecommand{\to}{\longrightarrow }

\newcommand{\inner}[1]{\left\langle #1 \right\rangle }

\def\[#1\]{\begin{align*}\begin{split} #1 \end{split}\end{align*} }
\def\AlB#1\AlE{\begin{align}\begin{split} #1 \end{split}\end{align} }

\renewcommand{\hat}{\widehat}

\title[Erratum for ``The Degree Theorem in higher rank'']{Erratum for {\em The Degree Theorem in Higher Rank}}
\author[Chris Connell]{Chris Connell$^\dagger$}
\thanks{$\dagger$ Supported in part by the Simons Foundation  and the CNRS}

\author[Benson Farb]{Benson Farb$^\ddagger$}
\thanks{$\ddagger$ Supported in part by the NSF}


\begin{document}

\begin{abstract}
The purpose of this erratum is to correct a mistake in the proof of Theorem 4.1 of \cite{CF}.
\end{abstract}

\maketitle

In this note we fix a mistake in Theorem 4.1 of \cite{CF}.  This error was pointed out to us by 
Inkang Kim and Sungwoon Kim, to whom we are extremely grateful.  The error occurs in the final step of the proof of Theorem 4.4 in \cite{CF}: the stated angle inequality should hold not just for a subspace $V$, but for each individual vector  in an orthonormal $k$-frame. 
The problem with the proof in \cite{CF} occurs at the very end of \S 5, at the top of page 52.  Lemma 5.4 in \cite{CF} applies to all of $V_k'$, but one needs to justify that this lemma applies to the subspace $W'$. 

\bigskip
\noindent
{\bf The fix. } Throughout the present paper we use the notation and terminology of  \cite{CF}.   The setup is as follows.  Let $X=G/K$ be a symmetric space of noncompact type with no local $\R, \mathbb{H}^2$ or $\SL_3(\R)/\SO(3)$ factors.  We also assume (cf.\  \S 4.2 of \cite{CF}) that $X$ is irreducible.  We fix a point $x\in X$ and a maximal flat $\mc{F}$ through $x\in X$.  The stabilizer of $x$ in $G$ is (after conjugation) $K$, and $K$ acts by the derivative action on the tangent space $T_xX$, which we identify as a 
subspace $\mathfrak{p}$ of the Lie algebra $\mathfrak{g}=\mathfrak{p}\oplus\mathfrak{k}$ of $G$, endowed with the standard inner product coming from the Killing form $B$.  
We identify $\mc{F}$ and $\mc{F}^\perp$ with their corresponding tangent spaces in $T_xX$.  As in \S 4.4 of \cite{CF}, define the {\em angle} 
between two subspaces $V,W\subseteq T_xX$  
as $$\angle (V,W):=\inf\{d_{\SO(T_xX)}(I,P): P\in \SO(H) \text{ with } 
PV\subset W\text{ or } PW\subset V\} $$

We do not see how to prove Theorem 4.4 of \cite{CF}, called there the ``Eigenvalue Matching Theorem'',  as stated.  We instead prove the following result.  
Call a set of vectors $\set{w_1,\dots,w_k}$ is a {\em $\delta$-orthonormal $k$-frame} if $\inner{w_i,w_j}<\delta$ for all $1\leq i<j \leq k$.
 
\begin{theorem}[Weak eigenvalue matching]
\label{theorem:evalue:matching}
For each symmetric space $X$ as above, there are constants $C_1$ and $C$ so that the following holds.  Given any $\epsilon<1/(\rank(X)+1)^2$,  for any orthonormal $k$-frame $v_1,\dots,v_k$ in
	$T_{x}X$ with $k\leq\rank(X)$, whose span $V$ satisfies $\angle(V,\mc{F})\leq \epsilon$ there is a $C_1\eps$-orthonormal $2 k$-frame given by vectors $v'_1,v''_1, \dots,v'_{k},v''_{k}$, such that for
	$i=1,\dots,k$: 	
	\begin{equation}
	 \label{eq:main1}
	\angle (hv'_{i},\mathcal{F}^\perp) \leq C\angle (hv_i,\mathcal{F})
	\end{equation}
	and
	\begin{equation}
	\label{eq:main2}
	\angle (hv''_{i},\mathcal{F}^\perp) \leq C\angle (hv_i,\mathcal{F})
	\end{equation}
	for every $h\in K$, where $hv$
	is the linear (derivative) action of $K$ on $v\in T_{x}X$.
\end{theorem}

Fortunately, Theorem~\ref{theorem:evalue:matching} is enough to deduce the main theorem (Theorem 4.1) of \cite{CF}. We make this deduction in \S\ref{section:Theorem4.1} below.

\medskip
\noindent
{\bf Acknowledgements. } We are extremely grateful to Inkang Kim and Sungwoon Kim for finding the mistake in \cite{CF}, and for their careful readings, questions and comments on the present paper. We are also grateful to the referee for making a number of corrections and suggestions that improved the paper.

\section{Proof of Theorem~\ref{theorem:evalue:matching} when $\epsilon=0$}
\label{section:proof:special}

By extending a given orthonormal $k$-frame to an orthonormal basis, it is enough to prove the theorem for the case $k=n$, where $n=\rank(X)$.    So let $\{v_1,\ldots ,v_n\}$ be an orthonormal $n$-frame in $T_xX$, and let $V$ denote its span.  In this section we prove will prove the following:

\medskip
{\it Theorem~\ref{theorem:evalue:matching} holds in the special case when $\epsilon=0$, that is when $V\subset \mc{F}$.  Further, in this special case, the theorem holds with only the assumption that $\set{v_i}$ spans all of $\mc{F}$, not necessarily that $\set{v_i}$ forms an orthonormal frame.}

\medskip

We thus assume throughout this section that $V\subset \mc{F}$.  Let $K_i$ denote the stabilizer of $v_i$.  Recall that the Lie algebra of $K_i$ is $\mf{m}\directsum_j \mf{k}_{\alpha_{i_j}}$, where $\mf{m}$ is the Lie algebra of the stabilizer of $\mc{F}$ and the sum is taken over the family of all one-dimensional spaces $\mf{k}_{\alpha_{i_j}}\subset \mf{k}$ such that $v_i$ belongs to the kernel of the (positive) root $\alpha_{i_j}$. We define for each $1\leq i\leq n$ the subspace
\[Q_i:=(\rm{span}\{K_i\cdot\mc{F}\})^\perp.\]

 For each positive root $\alpha$, we have $[\mf{k}_\alpha,\mf{a}]\subset \mf{p}_\alpha\directsum \mf{a}\subset \mf{g}_\alpha\directsum \mf{g}_{-\alpha}\directsum \mf{g}_0$.  In particular, $Q_i$ is spanned by the set of $\mf{p}_\alpha$ for those positive roots $\alpha\neq 0$ such that $\mf{k}_\alpha\not\in \mf{k}_{v_i}$.

 \begin{lemma}[{\bf Vectors in $Q_i$ satisfy \eqref{eq:main1}}]
 \label{lemma:qigood}
There exists a constant $C>0$, depending only on $\dim(X)$, so that for any $w\in Q_i$ and any $h\in K$: 

 \begin{equation}
  \label{eq:main5}
\angle (w,h\cdot\mathcal{F}^\perp) \leq C\angle (v_i,h\cdot\mathcal{F})
\end{equation}
where $h$ acts via the derivative action of $K$ on $v\in T_{x}X$.   
\end{lemma}

\begin{proof} This exact fact was proven in Lemma 5.3 of \cite{CF}: take $V:={\rm span}\{v_i\}$ and $V':={\rm span}\{w\}$ and apply the proof of that lemma verbatim starting with the line ``If no such constant $\ldots$.''  The earlier part of the lemma was meant only to produce such a $V'$. \end{proof}

Let $K_{\mc{F}}<K$ denote the stabilizer of $\mc{F}$, and let $m=\dim K/K_{\mc{F}}=\dim K-\dim K_{\mc{F}}$.  Note that $n+m=\dim(X)$.   For each positive root $\alpha$, let $\mf{p}_\alpha:=(\mf{g}_\alpha\directsum \mf{g}_{-\alpha})\cap\mf{p}$.  Now $[\mf{k}_\alpha,\mf{a}]\subset \mf{p}_\alpha\directsum \mf{a}\subset \mf{g}_\alpha\directsum \mf{g}_{-\alpha}\directsum \mf{g}_0$.   For each $i$ choose $a_i\in\mf{a}$ such that $b_i:=[k_i,a_i]$ spans $\mf{p}_{\alpha_i}:=(\mf{g}_{\alpha_i}\directsum \mf{g}_{-\alpha_i})\cap\mf{p}$.  Replacing $b_i$ by $b_i/||b_i||$, we can assume that each $b_i$ has length $1$.   Note that $b_i\not\in\mf{a}$ since $[\mf{k},\mf{a}] \cap \mf{a}=0$.

For two distinct roots $\alpha,\beta$ with 
$\alpha+\beta\neq 0$, the Killing form satisfies $B(\mf{g}_\alpha,\mf{g}_\beta)=0$.  It follows that $b_i\in\mc{F}^\perp$ and $\{b_i\}$ is orthonormal.  In particular, since this set 
has cardinality $\dim(\mc{F}^\perp)$, it forms an orthonormal basis for $\mc{F}^\perp$.  

Denote the Lie algebra of $K_i$ by  $\mf{k}_i$.  Now 
\[Q_i=({\rm span}(K_i\cdot \mc{F}))^\perp={\rm span}\{b_j: b_j\not\in [\mf{k}_i,\mf{a}]\}.\] 
Since $K_i$ is a proper subgroup of $K$, for each $j$ there exists $i$ so that $b_j\not\in [\mf{k}_i,\mf{a}]$; in particular $b_j\in Q_i$.   Note that this $i$ is not necessarily unique.  Thus, to summarize, the basis $\{b_i\}$ is adapted to the $Q_i$ in the sense that each $b_j$ belongs to some $Q_i$ and each $Q_i$ is spanned by the collection of $b_j$'s that it contains.

\begin{lemma}
\label{lemma:mainlemma}
There exists a subset of $\{b_i\}$ consisting of $2n$ distinct elements, two from each $Q_i, 1\leq i\leq n$.
\end{lemma}

We prove Lemma~\ref{lemma:mainlemma}  below.  Assuming this for now, let $v_i,v_i'$ denote 
the pair of vectors in $Q_i$ guaranteed by the lemma.  We claim that $\{v_1',v_1'',\ldots ,v_n',v_n''\}$ satisfy the conclusion of Theorem~\ref{theorem:evalue:matching}, thus proving that theorem in the special case $V\subset \mc{F}$, which we have assumed throughout this section.  To see this, first note that, by definition, the set $\{b_i\}$ is orthonormal, and its span is also orthogonal to $V$.  Since $v_i',v_i''\in Q_i$,  the inequality \eqref{eq:main5} of Lemma~\ref{lemma:qigood} gives exactly inequalities \eqref{eq:main1} and \eqref{eq:main2} of the theorem, as desired.  The rest of this paper 
is devoted to proving Lemma~\ref{lemma:mainlemma}.

\subsection{Combinatorial Translation}

To prove Lemma~\ref{lemma:mainlemma} we first translate it into a problem that is purely 
combinatorial.  To this end, let $A=A_G=(a_{ij})$ be the $n\times m$ matrix  whose $i,j$ entry is $1$ if $b_j$ belongs to $Q_i$ and $0$ otherwise.  Lemma~\ref{lemma:mainlemma} is then the statement that we can pick two $1$ entries from each row of $A$, so that all of our $2n$ choices are in different columns. More formally:

\bigskip
\noindent
{\bf Key Claim (Lemma~\ref{lemma:mainlemma} restated): } {\it For each $1\leq i\leq n$ there exists $1\leq j_i,k_i\leq m$ with $j_i\neq k_i$ so that each $a_{ij_i}=a_{ik_i}=1$ and 
$\bigcup_{i=1}^n \set{j_i,k_i}$ has cardinality $2n$.}

\smallskip
Given the Key Claim, we set $v_i':=b_{j_i}$ and $v_i'':=b_{k_i}$,  proving (as explained above) Theorem~\ref{theorem:evalue:matching}. The rest of this paper is devoted to proving the Key Claim.

Note that the Key Claim is true for $A$ if and only if it is true for any matrix obtained from $A$ by permuting its rows or columns, as these operations correspond to just re-ordering the $Q_i$'s and $b_j$'s, respectively.  We think of the $n\times m$ matrix of of being a list $u_1,\ldots ,u_n$ of $n$ row vectors, each in $\{0,1\}^n$.    Before proving properties of the matrix $A$, we will need the following lemma from Lie theory.

\begin{lemma}[{\bf Codimension of proper Lie subgroups of $K$}]
\label{lemma:numbers}
Let $G\neq \SL(3,\R)$ be a connected, simply-connected, simple Lie group with 
$n:=\rank_\R(G)\geq 2$.  Let $K$ denote the maximal compact subgroup of $G$.  
Let $H<K$ be the stabilizer of a vector in $\mc{F}$, and let $d(H):=\dim K -\dim H$. Then:
\begin{enumerate}
\item If $K=\SO(n+1)$ then either $d(H)\geq 2n-2$ or $d(H)=n$ and $H$ is locally isomorphic to $\SO(n)$.
\item If $K=\SO(n)\times \SO(n+r)$ with $r\geq 0$ then $d(H)\geq 2n-2+r$. 
\item For all other $K$ we have $d(H)\geq 2n-1$.
\end{enumerate}
\end{lemma}

Lemma~\ref{lemma:numbers} should not be a surprise since the dimension of a rank $n$ compact Lie group $K$ grows quadratically in $n$, and typically the rank of a proper Lie subgroup of $K$ has rank $<n$, and $(n+1)^2-n^2=2n+2$.

\begin{proof}
The list of maximal compact subgroups $K$ of all possible real and complex, connected, simply-connected, simple  Lie groups, including exceptional groups, is given on pages 684--718 of \cite{Kna}. Since we are bounding codimension from below, we can assume that $H$ is a subgroup of a maximal proper subgroup of $K$. 

For each of the $K$ coming from the classical algebras, the list of possible connected Lie subgroups is given in Tables 5--8 found on pages 1018--1027 of \cite{AFG}. (For a simpler list that is sufficient in our case, the maximal Lie subalgebras are found in Tables 1--4 on pages 987--1010 of that same paper.) The list of the exceptional cases can be found in Table 1.7 on page 37 of \cite{Ant}, where one interprets the list respectively as either the complex or real compact form, ignoring the noncompact real split cases. 

The dimension of each of these groups is the sum of the dimensions of its simple factors.  The dimensions for these can be computed for example, from Table 1 on page 66 of \cite{Hum}. (Note that for the complex groups, to compute the real dimension one must multiply the number of positive roots by two and add the rank.) The possible dimensions of maximal compact lie subgroups follows by going through each case and plugging in the numbers from these tables.

Now there is one general case where the codimensions of maximal proper subgroups do not agree with the codimension bounds listed in the statement, namely there are maximal subgroups of $K=SO(n)\times SO(n+r)$ which are locally isomorphic to $SO(n-1)\times SO(n+r)$ which only have codimension $n-1$ in $K$. (This case arises for $G=SO(n,n+r)$.) However, any maximal stabilizer of a vector in $\mc{F}$ must be locally isomorphic to a subgroup of $SO(n-1)\times SO(n+r-1)$ since neither factor group stabilizes any vector of $\mc{F}$ and maximal proper subgroups of $SO(k)$ for any $k>1$ are locally isomorphic to subgroups of $SO(k-1)$. Hence the codimension of $H<K$ is $2n-2+r$ in this case.

\end{proof}

We remark that the data we need from all these tables is really originally due to Dynkin \cite{Dyn,Dyn2} who computed these in the algebraically closed case. However, it is a tedious exercise to extract all of the real reductive, split and compact form cases that arise. This has been done for us in the more modern references cited above.

With Lemma~\ref{lemma:numbers} in hand, we are now ready to prove properties of the matrix $A$.  
Recall that we have reduced the situation to the case where $X$ is irreducible. Thus we only care about simple Lie group $G$ not locally isomorphic to $\SL_3(\mathbb{R})$.  Let $|u_i|$ denote the number of $1$ entries of $u_i$.

\begin{lemma}[{\bf Properties of \boldmath$A$}]
\label{lemma:overlap}
Let $G\neq \SL_3(\R)$ be a connected, simply-connected, simple Lie group with $n:={\rm rank}_{\R}(G)\geq 2$.  Let $K$ denote the maximal compact subgroup of $G$.  Let $A=A_G$ be defined as above, with row vectors $u_1,\ldots ,u_n$.  Then the following hold.
\begin{enumerate}
\item Each column of $A$ has at least one entry equal to $1$.
\item $|u_i|\geq n$ for each $i$.
\item If $K$ is not locally isomorphic to $\SO(n+1)$, then $|u_i|\geq 2n-2$ for each $i$. 

\item If $u_i=u_j$ then $|u_i|=|u_j|\geq 2n-1$.
\item For $i\neq j$, if $|u_i|<2n-2$ and $|u_j|<2n-2$, then there is at most one $k$ with $a_{ik}=a_{jk}=1$. 
\end{enumerate}
\end{lemma}

\begin{proof}
 If (1) does not hold, then there is some $b\in \{b_i\}$ that does not lie in $Q_i$ for any $i$.  We can write $b=[k,a]$ where $k$ (resp.\ $a$) is a positive root vector in $\mathfrak{k}$ (resp.\  $\mathfrak{a}$).  Let $\mathfrak{k}_i$ be the Lie algebra of $K_i$.  

Since $b\in \{b_i\}$ but 
$b\not\in Q_i=\{{\rm span}K_i\cdot \mc{F}\}^\perp$, the fact that $\{b_i\}$ is an orthonormal basis for $\mc{F}^\perp$ implies that 
$b\in {\rm span}\{K_i\cdot\mc{F}\}$, so we can write 
$b\in[\mathfrak{k}_i,\mathfrak{a}]$.   
Since this is true for each $i$, it follows that 
$k\in \cap_i\mathfrak{k}_i$. However, since the entire frame $\set{v_i}$ forms a basis for $\mc{F}$ and $\exp k$ stabilizes each vector simultaneously, $k$ belongs to the stabilizer of $\mf{a}$ in $\mf{k}$. In other words, $k\in \mf{k}\cap \mathfrak{g}_0$, contradicting the fact that $k$ belongs to a positive root space. Thus it must be that each $b_j$ lies in some $Q_i$, proving (1).

To prove (2), we first note that $X=K\cdot \mc{F}$.  Thus  
\[\dim Q_i=\dim {\rm span}(\{K_i\cdot\mc{F}\})^\perp=\dim(X)-\dim{\rm span}\{K_i\cdot\mc{F}\}\]
Since $K_{\mc{F}}$ is an extension of the pointwise stabilizer $K'_\mc{F}$ of $\mc{F}$ 
by a finite group (the Weyl group of $G$), and so $\dim K_{\mc{F}}=\dim K'_{\mc{F}}$, we also have
\[\dim X=\dim {\rm span}\{K\cdot \mc{F}\}=\dim K+\dim\mc{F}-\dim K_{\mc{F}}\]
and
\[\dim {\rm span}\{K_i\cdot \mc{F}\}=\dim K_i+\dim\mc{F}-\dim K_\mc{F}\]
since $K_\mc{F}$ is contained in $K_i$. Combining the above three equations gives
\begin{equation}
\label{eq:dimqi}
\dim Q_i=\dim K - \dim K_i
\end{equation}

Items (2) and (3) now follow by applying Lemma~\ref{lemma:numbers}.

We now prove (4).  If $u_i=u_j$ then $K_i=K_j$ and therefore $v_i$ and $v_j$ belong to the same singular subspace $W\subseteq \mc{F}$, and neither lies in a more singular subspace. Hence $\dim(W)\geq 2$, and so $W$ contains a $1$-dimensional subspace that is $K'$-invariant for some proper subgroup $K'$ of $K$ properly containing $K_i$.  Lemma~\ref{lemma:numbers}  implies that if $K'$ does not already have codimension at least $2n-1$, then either $K$ is isomorphic to $SO(n+1)$ and $K'$ is necessarily locally isomorphic to $SO(n)$ or else $K=SO(n)\times SO(n)$ and $K'$ is locally isomorphic to $SO(n-1)\times SO(n-1)$. In the second case any proper Lie subgroup of $K'$ already has codimension $2n-1$ in $K$. In the first case, $K_i$ can have dimension no larger than that of $SO(n-1)$, corresponding to a proper Lie subgroup of $K'$. Therefore $K_i$ has codimension at least $n-1$ in $K'$ and codimension at least $2n-1$ in $K$, as indicated. This proves (4).

For (5), we first note that, by Proposition 2.20.5  of \cite{Ebe}, both  $K_i$ and $K_j$ are proper semisimple subgroups of $K$.   From Lemma~\ref{lemma:numbers}, we note that the only case where there can exist $u_r$ with $|u_r|< 2n-2$ is when $K$ is locally isomorphic to $SO(n+1)$ and $K_r$ is locally isomorphic to $SO(n)$.   We assume this is the case, and hence $K_i$ and some $K_j$ are locally isomorphic to $SO(n)$.
   

From the discussion above, it remains to show that the dimension of $Q_i\cap Q_j$, or the intersection of $\mf{k}\ominus\mf{k}_{i}$ and $\mf{k}\ominus\mf{k}_j$, is at most one. We note that the basis of $Q_i$ consists of all of those $b_r$'s whose corresponding $k_r\in \mf{k}$ is not in $\mf{k}_i$; similarly, $Q_j$ consists of those $b_r$ for which $k_r\not\in\mf{k}_j$.  Hence $Q_i+Q_j$ is spanned by the set of those $b_r$ with $b_r\not\in\mf{k}_i\cap \mf{k}_j$. This corresponds exactly to $((K_i\cap K_j)\cdot \mc{F})^\perp$. 

The dimension of $Q_i+Q_j$ is therefore the codimension in $K$ of $K_i\cap K_j$.  The subgroups $K_i$ and $K_j$ are distinct by (4). The intersection of two distinct copies of $\mf{so}(n)$ in $\mf{so}(n+1)$ is isomorphic to a subalgebra of $\mf{so}(n-1)$. Hence $\dim (Q_i+Q_j) \geq 2n-1$ and hence $\dim Q_i\cap Q_j=\dim Q_i+\dim Q_j-\dim (Q_i+Q_j)\leq 2n-(2n-1)=1$, completing the proof.

\end{proof}

\subsection{Solving the combinatorial problem}

By re-ordering and relabeling the rows, we can and will assume $|u_i|\leq |u_{i+1}|$ for all $i$.  Lemma~\ref{lemma:overlap}(2), Lemma~\ref{lemma:overlap}(3) and Lemma~\ref{lemma:numbers}(1) together imply that for each $i$ either $|u_i|=n$ or else $|u_i|\geq 2n-2$.

We will now describe an algorithm which takes input a subset of row vectors $\{u_i\}$, and at each stage removes one of the vectors, and changes each vector by removing two of its entries. We still call the remaining vectors by the same names $u_i$.   First consider the case that there exists $p>0$ so that $|u_i|=n$ for each $1\leq  i\leq p$.   Set $t=1$.  For any $t$ let $N(i,t)$ denote the number of $1$'s left in the vector $u_i$ at the start of Stage $t$ of the algorithm.  So for example $N(i,1)=n$ for all $1\leq i\leq p$.  Now, starting with $t=1$,  perform ``Stage $t$'' of the following algorithm on the row vectors $\{u_1,\ldots ,u_p\}$:  

\begin{enumerate}
\item[Step 1: ] Re-order the rows so that $N(i,t)\leq N(i+1,t)$ for each $1\leq i\leq p-t$.

\item[Step 2: ] Choose two $1$ entries of the top row; and let $j_t,k_t$ be the column numbers of these two entries. 

\item[Step 3: ] Delete the top row and the columns $j_t$ and $k_t$, still calling the remaining vectors $u_j$ by their original name.  Now increase the counter $t$ by $1$, and go to Step 1.
\end{enumerate}

   At each stage we remove two columns corresponding to the columns of two $1$ entries of the top row. By Lemma~\ref{lemma:overlap}(5), this implies that at most one $1$ is removed from any of the other rows.  We thus have that  if the vector $u_i$ remains at Stage $t$ then 
\begin{equation}
\label{eq:algo1}
N(i,t)\geq N(i,t-1)-1
\end{equation}

The algorithm can only fail at Step $2$. Let $d$ be the smallest $t$ for which Stage $t$ of the algorithm fails.  $A$ has at most $n$ rows, so $d\leq n$.  Let $u_j$ denote the top row after performing Step $1$ at Stage $t=d$.  The assumption of failure is then $N(j,d)\leq 1$.  Lemma~\ref{lemma:overlap}(5)
implies that at each stage $t=1,\ldots ,d-1$, at most one $1$ was removed from $u_j$.   But $N(j,1)\geq n$, so that 
\[1\geq N(j,d)\geq n-(d-1)=n-d+1\]
and so $d\geq n$, so that $d=n$ and $N(j,n)=1$.  In other words, the algorithm will succeed 
in choosing two $1$'s from each row $u_i$ with $|u_i|=n$, except possibly if  $|u_i|=n$ for each $1\leq i\leq n$, in which case the algorithm can only possibly fail at Stage $n$, at the final row vector $u_j$.  

If $N(j,n)=2$ we are done, so assume $N(j,n)\leq 1$.  Application of \eqref{eq:algo1} gives $N(j,n-1)\leq 2$.  By our ordering in Step 1, the top row $u_k$ at Stage $t=n-1$ has $N(k,n-1)\leq 2$.  A repeated application of \eqref{eq:algo1} gives that $N(i,2)\leq n-1$ for all $2\leq i\leq p$.  This means that each $u_2,\ldots ,u_p$ must have a $1$ in one of the two columns removed from $u_1$ during Step 2 of Stage 1.  Since $N(1,1)=n>2$, there exists an entry of $u_1$ that does not overlap with any other $u_j$. Instead of choosing the entry of $u_1$ that overlaps with $u_j$, choose this entry to remove. Then we can choose the original entry from $u_j$, so that we do not fail at the last stage.

We have thus shown that the above algorithm always succeeds: we can choose two $1$'s from each $u_i$ with $|u_i|=n$, all satisfying the Key Claim.    Since we are done if every row is of this form, we can now assume that $|u_n|\geq 2n-2$.   Note that it may be that $|u_i|=2n-2$ for each $1\leq i\leq n$.

Having performed the algorithm successfully on the (possibly empty) $\{u_i: |u_i|=n\}$, we now continue with the algorithm on the remaining vectors $\{u_i: |u_i|\geq 2n-2\}$, not resetting $t=1$.   Since at most two columns are removed at any stage of the algorithm, the only way for the algorithm to fail with some vector $u_j$ with $|u_j|=2n-2$ or $|u_j|=2n-1$ in the top row is at Stage $t=n$.  If this happens then of the $2n-2$ columns removed in the first $n-1$, stages, at least $2n-3$ of them must have been in columns in which $u_j$ has a $1$.  

Lemma~\ref{lemma:overlap}(1) states that each column of $A$ has a $1$.  The number of columns of $A$ is the dimension of the symmetric space $X$ minus the rank. From Table II on page 354 of \cite{Hel}, we see that for any given rank $n\geq 2$ the $\dim(X)-n$ is at least $n(n+1)/2$, equality occurring only for the case when $G=\SL(n+1,\R)$. Hence the total number of columns of $A$ is always at least $n(n+1)/2$.  Thus, after having removed at most $2n-2$ columns, there must be at least $n(n+1)/2-(2n-2)= (n^2-3n+4)/2\geq 2$ (for $n\geq 3$ - this is where we are using the hypothesis $G\not\approx\SL_3(\R)$) columns not yet removed, which have an entry with $1$.  Call two of these columns $c_1,c_2$.  These $1$ entries are entries in row vectors $u_p,u_q$ for some $p,q\neq j$, with $p=q$ possible.  At some stage $u_p$ was the top row, and two columns were removed corresponding to two $1$ entries of $u_p$.   Put back one of these columns and remove $c_1$ instead.  Do the same thing with $u_p$ replaced by $u_q$ and $c_1$ replaced by $c_2$.    

We claim that there is not a failure at stage $n$, with row vector $u_j$. If $N(j,n)=0$ then precisely two $1$'s from $u_j$ were removed at each stage $1,\ldots ,n-1$, so that the two columns we just replaced now each give a $1$ back to $u_j$, so that the algorithm doesn't fail at $u_j$.  If $N(j,n)=1$ then it is still the case that of the $2n-2$ columns removed, at most one such column of $u_j$ did not have a $1$ entry.  In particular $u_j$ had a $1$ removed from one of the columns $c_1$ or $c_2$. Since we replaced this column, and since $N(j,n)=1$, the replacement gives two $1$ entries for $u_j$, and again the algorithm does not fail at stage $n$.

We have thus shown that the modified algorithm given above terminates with the choices proving the Key Claim.

\section{Finishing the proof of Theorem~\ref{theorem:evalue:matching}}

In this section we complete the proof of prove Theorem~\ref{theorem:evalue:matching}.  

Let $k_i\in K$ be a closest element to the identity such that ${\hat w_i}:=k_i^{-1} v_i$ lies in $\mc{F}$. If it happens that there is a more singular vector in $\mc{F}$ very nearby to $\hat{w_i}$, then it may be that $k_i$ could be large (say on the order of $\pi$) as it moves $\hat{w_i}$ through a large rotation around the singular vector, but keeping it very close to $\mc{F}$. Hence we begin by replacing $\hat{w_i}$ with the most singular vector $w_i$ in the ball of radius $\eps_o=1/(\rank(X)+1)^2$ about $\hat{w_i}$ and that is closest to $\hat{w_i}$. (This vector will be unique as the singular subspaces form linear flags.) By the choice of $\eps_o$, the new $w_i$ will be the most singular in its $\eps_o$-ball. 


Set $K_i={\rm Stab}_K(w_i)$ and note that $K_i$ will contain the stabilizer of 
$\hat{w_i}$. Since each element $k$ of a stabilizer subgroup not belonging to $K_i$ 
stabilizes a vector at least $3\eps_o$ away from $\hat{w_i}$, it follows that $k$ moves 
$\hat{w_i}$ at least a distance of $\frac{\displaystyle \eps_o}{\displaystyle 4\pi}d_K(k,1)$ away from $\hat{w_i}$, 
provided $d_K(k,1)<\frac{\pi}{4}$. 

We will show that there is a small element of $K$ that moves $v_i$ into $K_i\cdot 
\mc{F}$, as follows.  Since the derivative at $0$ of the exponential map $\exp:\mf{g}\to 
G$ is the identity map, we can transport metric estimates to $\mf{g}$. Therefore, 
setting $\hat{a}_i$ to be the lift to $\mf{a}$ of $\hat{w_i}$, there is a $c_o$ 
depending only on $\eps_o$ such that each element $u\in \mf{k}$ orthogonal to $\mf{k}_i$ 
and with $|u|<1$ has $|[u,\hat{a}_i]|\geq c_o|u|$. In particular, the 
$\frac{\eps}{c_o}$-neighborhood $U$ of $0$ in $\mf{k}$ has the property that 
$[U+\mf{k}_i,\mf{a}]+\mf{a}$ contains the $\eps$-neighborhood of $\hat{a_i}$ in 
$\mf{p}$. Consequently, descending back to $X$, there is a constant $c_1$ depending only on 
$\eps_o$ (or equivalently $\rank(X)$) such that smallest element $k_i'=\exp(u)\in K$ 
such that $v_i\in k_i' K_i\cdot \mc{F}$ has $d_{K}(k_i',1)<c_1\eps$. 

We also have $\angle(w_i,v_i)<\eps_0+\eps<2\eps_0$.  Since $\{v_i\}$ is orthonormal, it 
follows that $\{w_i\}$ is $4\epsilon_0$-orthonormal, and in particular it is still a 
frame.

Since $\{w_i\}\subset\mc{F}$, we can apply the special case $\epsilon=0$ of Theorem \ref{theorem:evalue:matching} proved in \S\ref{section:proof:special}. (Recall that for this case, we did not require the $\set{w_i}$ to be orthonormal.) This produces 
an orthonormal (since $\epsilon=0$) $2k$-frame $\set{w_i',w_i''}$ satisfying the angle 
inequalities of Theorem~\ref{theorem:evalue:matching} with $v_i,v_i',v_i''$ replaced by 
$w_i,w_i',w_i''$.  (Observe that $w_i'$ and $w_i''$ also satisfy the angle inequalities 
with $w_i$ replaced by $\hat{w_i}$ as well since $w_i'$ and $w_i''$ are orthogonal to all 
of $K_i\cdot\mc{F}$ and $K_i$ contains the stabilizer of $\hat{w}_i$.)

  Moreover, as proved in the $\epsilon=0$ case of Theorem \ref{theorem:evalue:matching}, 
  $w_i',w_i''\in (K_i\mc{F})^\perp$ for each $i$.  Now let $v_i'=k_i'w_i'$, let  
  $v_i''=k_i'w_i''$ and let $z_i=v_i'-w_i'$.    Since $d_K(k_i',1)<c_1\eps$ it follows 
  that $|z_i|<c_1\eps$ and 
  \[\begin{array}{ll}
  |<v_i',v_j'>|&=|<w_i'+z_i,w_j'+z_j>|\\
  &=|0+<w_i',z_j>+<z_i,w_j'>+<z_i,z_j>|\\
  &\leq 3c_1\eps
  \end{array}\] for all $1\leq i,j\leq k$. The same bound holds for $<v_i',v_j''>$ and 
  $<v_i'',v_j''>$ by the same computation.  Now set $C_1:=3c_1$. (Note that $v_i'$ and 
  $v_i''$ are also orthogonal to $v_i$ since $k_i'(K_i\mc{F})^\perp=(k_i' 
  K_i\mc{F})^\perp$.) 
  
  Finally $\angle (hv'_{i},\mathcal{F}^\perp)=\angle (hk_i'w'_{i},\mathcal{F}^\perp) 
  \leq C\angle (hk_i' \hat{w}_i,\mathcal{F})=C\angle (hv_i,\mathcal{F})$, and similarly 
  for $v_i''$. This completes the proof of 
  Theorem~\ref{theorem:evalue:matching}.

\section{Proving Theorem 4.1 of \cite{CF}}
\label{section:Theorem4.1}

In this section we prove the main theorem (Theorem 4.1) of \cite{CF}.  The proof as given in \S 4.5 of \cite{CF} needs to be slightly modified, given that we do not know Theorem 4.4 of \cite{CF} as stated, but only the slightly weakened form, Theorem~\ref{theorem:evalue:matching} above.

On page 41 of \cite{CF} we choose $\epsilon=1/(\rank(X)+1)$.  We now instead choose 
$\epsilon$ so small that $\epsilon<1/(\rank(X)+1)^2$ and so that for any $t$, when $\sin(t)<\epsilon$ then $\sin(t)>t/2$. This new choice of constant of course still depends only on $\rank(X)$, and the only affect of this change will be to change the resulting constants in the proof of the theorem.  
As in \cite{CF}, we let $L_1,\ldots ,L_k$ be the $k\leq \rank(X)$ eigenvalues of the positive semi-definite quadratic form $Q_2$ that are strictly less than $\epsilon$.   
As stated in \cite{CF}, if no such eigenvalues exist then we are done, so we assume $k\geq 1$.  Label the $L_i$ so that $0\leq L_1\leq\cdots \leq L_k$.  Denote by $v_i$ the eigenvector associated to $L_i$. 

Plugging the formula $r(v)=\sin^2\angle(v,\mc{F})$, given on page 42 of \cite{CF}, into the formula for $L_i$ given on the last line of page 41 of \cite{CF}, gives

\[
L_i=\int_{\pa_F X}\sin^2\angle(kv_i,\mc{F})d\sigma_y^s(k).
\]

Recall that we are identifying $\pa_FX$ with $K/M$, whose elements we write as elements of $K$, remembering that they are really equivalence classes.  For each $i$ let 
\[A_i:=\{k\in\pa_FX: \sin^2(\angle kv_i,\mc{F})\leq \sqrt{L_i}\}\] 
and let $B_i:=\pa_FX-A_i$.  

We claim that for each fixed $i$, each $k\in A_i$ moves $v_i$ a small angle from $\mc{F}$.  To see this, note that for any $k\in A_i$: 
\[\sqrt{L_i}\geq \sin^2\angle(kv_i,\mc{F})>(\angle(kv_i,\mc{F})/2)^2\]
so that $\angle(kv_i,\mc{F})\leq 2(L_i)^{1/4}$, as desired. Here we have used our choice of $\epsilon$ to obtain the second inequality.

Now
\[L_i=\int_{\pa_F X}\sin^2\angle(kv_i,\mc{F})d\sigma_y^s(k) \geq \int_{B_i}\sin^2\angle(kv_i,\mc{F})d\sigma_y^s(k)\geq \sqrt{L_i}\cdot\sigma_y^s(B)\]
so that  $\sigma_y^s(B_i)\leq \sqrt{L_i}\leq \sqrt{\epsilon}$  for each $i$.  Since we have chosen $\epsilon<1/(\rank(X)+1)^2$, we obtain
\[\sigma_y^s(B_1\cup\cdots\cup B_k)\leq k\sqrt{\epsilon}\leq \rank(X)\sqrt{\epsilon}<\frac{\rank(X)}{\rank(X)+1}<1.\]
Since by definition $A_i$ is the complement $B^c_i$ and $\sigma_y^s(\pa_F(X))=1$, it follows that $A:=A_1\cap\cdots \cap A_k\neq\emptyset$.  Any element in $A$ moves all of $V=\op{span}\set{v_1,\dots,v_k}$ to within $2\eps^{1/4}$ of $\mc{F}$.

We have just proved that there exists an element $k_0\in A$ with the property that $\angle(k_0v_i,\mc{F})\leq 2\epsilon^{1/4}$ for each $i$.  Now apply Theorem \ref{theorem:evalue:matching} to $\{k_0v_i\}$, and note that the vectors $\set{v_i',v_i''}$ produced satisfy the same inequalities with $k_0v_i$ replaced by $v_i$.   We now apply these inequalities in the string of inequalities starting on line 2 of Page 43 of \cite{CF} , with only one modification, namely, the first line should now read:

\[
\det Q_1\leq C' \prod_{i=1}^k \inner{v_i'Q_1,v_i'}\inner{{v_i''Q_1,v_i''}}
\]
Where $C'=\frac{1}{(1-C_1\eps)^{4k}}$ and $C_1$ is from Theorem \ref{theorem:evalue:matching}. This uniform constant will also be carried along in the rest of the inequalities and then absorbed into the final constant $C$.

\end{document}